\documentclass[letterpaper, 10 pt, conference]{cls/ieeeconf}
\IEEEoverridecommandlockouts
\usepackage[table,usenames,dvipsnames]{xcolor}      
\usepackage[noadjust]{cite}
\usepackage{amsmath,amssymb,amsfonts,amsthm,dsfont,mathtools}
\usepackage{nicematrix}

\usepackage{adjustbox}
\usepackage{graphicx,tabularx,adjustbox,color}
\usepackage{multirow}
\usepackage[font=footnotesize]{caption}
\usepackage[font=footnotesize]{subcaption}

\usepackage{algorithmic}
\usepackage[ruled,vlined]{algorithm2e}
\usepackage{stackengine}

\usepackage{dirtytalk}
\allowdisplaybreaks
\renewcommand{\baselinestretch}{.985}

\usepackage[breaklinks=true, colorlinks, bookmarks=true, citecolor=Black, urlcolor=Violet,linkcolor=Black]{hyperref}

\newtheorem{lemma}{Lemma}
\newtheorem{theorem}{Theorem}
\theoremstyle{definition}

\newtheorem*{problem*}{Problem}

\newtheorem*{remark*}{Remark}
\newtheorem{assum}{Assumption}

\newcommand{\ba}{\begin{align}}
\newcommand{\ea}{\end{align}}
\newcommand{\fr}{\frac}

\newcommand{\ep}{\varepsilon}

\title{\LARGE \bf Event-Triggered Safe Stabilizing Boundary Control for the Stefan PDE System with Actuator Dynamics} 

\author{
Shumon Koga$^1$\thanks{$^1$Department of Electrical and Computer Engineering, UC San Diego, 9500 Gilman Drive, La Jolla, CA, 92093-0411, {\tt\small skoga@ucsd.edu}}, 
Cenk Demir$^2$, 
and Miroslav Krstic$^2$ \thanks{$^2$Department of Mechanical and Aerospace Engineering, UC San Diego, 9500 Gilman Drive, La Jolla, CA, 92093-0411, {\tt\small cdemir@ucsd.edu, krstic@ucsd.edu}}}

\begin{document}
\maketitle
\begin{abstract}
This paper proposes an event-triggered boundary control for the safe stabilization of the Stefan PDE system with actuator dynamics. The control law is designed by applying Zero-Order Hold (ZOH) to the continuous-time safe stabilizing controller developed in our previous work. The event-triggering mechanism is then derived so that the imposed safety conditions associated with high order Control Barrier Function (CBF) are maintained and the stability of the closed-loop system is ensured. We prove that under the proposed event-triggering mechanism, the so-called ``Zeno" behavior is always avoided, by showing the existence of the minimum dwell-time between two triggering times. The stability of the closed-loop system is proven by employing PDE backstepping method and Lyapunov analysis. The efficacy of the proposed method is demonstrated in numerical simulation. 
\end{abstract}

\section{Introduction}
\label{sec: intro}

Safety is an emerging notion in control systems, ensuring required constraints in the state, which is a significant property in numerous industrial applications including autonomous driving and robotics. Classically, such a constrained control design has been treated by model predictive control \cite{hewing2020learning}, reference governor \cite{gilbert2002nonlinear}, reachability analysis \cite{bansal2017hamilton}, etc. Following the pioneering work by Ames, et. al \cite{ames2016control}, the concept of Control Barrier Function (CBF) and the safe control design by CBF using Quadratic Programming (QP) have been widely spread to the control  community, such as robust CBF \cite{JANKOVIC2018359}, adaptive CBF \cite{xiao2021adaptive}, fixed-time CBF \cite{garg2021robust}. The system is said to be safe if the required safe set is forward invariant, guaranteeing the positivity of CBF. 

While most of the research in safe/constrained control has focused on the systems described by Ordinary Differential Equations (ODEs), a few recent works have been developing and analyzing the safety or state constraints in the systems described by Partial Differential Equations (PDEs), where the safety in the infinite-dimensional state needs to be satisfied, such as for distributed concentration \cite{refId0}, gas density \cite{karafyllis2021global}, or the liquid level \cite{karafyllis2021spillfree}. Our recent work \cite{koga2022ACC} has incorporated the concept of CBF into the boundary control of a PDE system, the so-called "Stefan system" \cite{Shumon19journal,KKbook2021}, which is a representative model for the thermal melting process \cite{koga2019arctic} and biological growth process \cite{demir2021neuron}. We have designed the nonovershooting control \cite{krstic2006nonovershooting} to achieve both safety and stabilization of the system around the setpoint, and also developed a CBF-QP safety filter for a given nominal control input. The proposed control law is given in continuous time; however, many practical control systems do not afford sufficiently high frequency in the actuator to justify the continuous-time control input.

Some practical technologies own  constraints in sensors and systems with respect to energy, communication, and computation, which require the execution of control tasks when they are necessary \cite{6425820}. To deal with this problem, the digital control design should be aimed to reduce the number of closing the loop based on the system state, developed as event-triggered control. Pioneering work on event-triggered control is designed for PID controllers in \cite{AARZEN19998687}. Following the literature, authors of \cite{1184824} have demonstrated the advantages of event-driven control over time-driven control for stochastic systems. The authors in \cite{4303247} proposed event-triggered scheduling to stabilize systems by a feedback mechanism. Subsequent works  \cite{heemels2008analysis,kofman2006level} proposed novel state feedback and output feedback controllers for linear and nonlinear time-invariant systems. Unlike the former literature, \cite{6160794} and \cite{donkers2010output} have regarded the closed-loop system under the event-triggered control as a hybrid system and guaranteed asymptotical stability with relaxed conditions for the triggering mechanism. To utilize the effect of this relaxation, dynamic triggering mechanisms are presented in \cite{girard2014dynamic}. 

While all the studies of the event-triggered control mentioned above are for ODE systems, authors of \cite{selivanov2016distributed} and \cite{yao2013resource} proposed event-triggered in-domain control for hyperbolic PDE systems, and authors of \cite{espitia2016event} proposed event-triggered boundary control for hyperbolic PDE systems. Following these studies, \cite{espitia2021event} developed an event-triggered boundary controller for a reaction-diffusion PDE, and \cite{wang2021adaptive} derived an adaptive event-triggered control for a hyperbolic PDE with time-varying moving boundary with bounded velocity. As a state-dependent moving boundary PDE, an event-triggered control has been developed to stabilize the one-phase Stefan PDE in \cite{rathnayake2022event}. However, the event-triggering mechanism in \cite{rathnayake2022event} owns a limitation that the upper bound of the dwell-time is identical to the sole condition for sampling-time scheduling developed in \cite{koga2021towards}. Indeed, the condition of the sampling schedule in \cite{koga2021towards} serves as the necessary condition for ensuring the required constraint in the Stefan system. While, such a necessary condition has not been clarified in the Stefan system with actuator dynamics considered in \cite{koga2022ACC}, which requires the methodology from high-relative-degree CBF for satisfying the state constraint.  

The event-triggering mechanism has been developed for the purpose of safety with CBF-based methods in  \cite{yang2019self}, which proves the nonexistence of the so-called ``Zeno behavior" by showing the existence of a lower bound of the dwell-time. Following this work, \cite{taylor2020safety} achieves safety under the event-triggering mechanism by ensuring the existence of the minimum bounded interevent time using input-to-state safe barrier functions. The recent study in \cite{long2022safety} has proposed the safety-critical event-triggered controller design for general nonlinear ODE systems. In addition, the event-triggered control with high-order CBFs under unknown system dynamics has been handled in \cite{xiao2021event} by adaptive CBF approach and in \cite{dhiman2021control} by Gaussian Process learning approach. Such digital safe control methods have been applied to spacecraft orbit stabilization \cite{ong2022stability}, network systems \cite{ong2021performance}, and so on. However, the safe event-triggered control for PDE systems has not been established yet.   

The contributions of the paper include (i) designing the event-triggered mechanism for the Stefan PDE system with actuator dynamics, so that both the safety and stability are maintained, (ii) and proving the nonexistence of Zeno behavior by showing the existence of the minimum dwell-time between two triggering times. Indeed, this paper provides the first study of the event-triggered boundary control for a PDE system to achieve safety and stability. The safety is shown by employing the high order CBF, and the stability is proven by utilizing PDE backstepping method and Lyapunov analysis. 
\section{Stefan Model and Constraints}\label{sec:problem}

\begin{figure}[t]
\centering
\includegraphics[width=0.99\columnwidth]{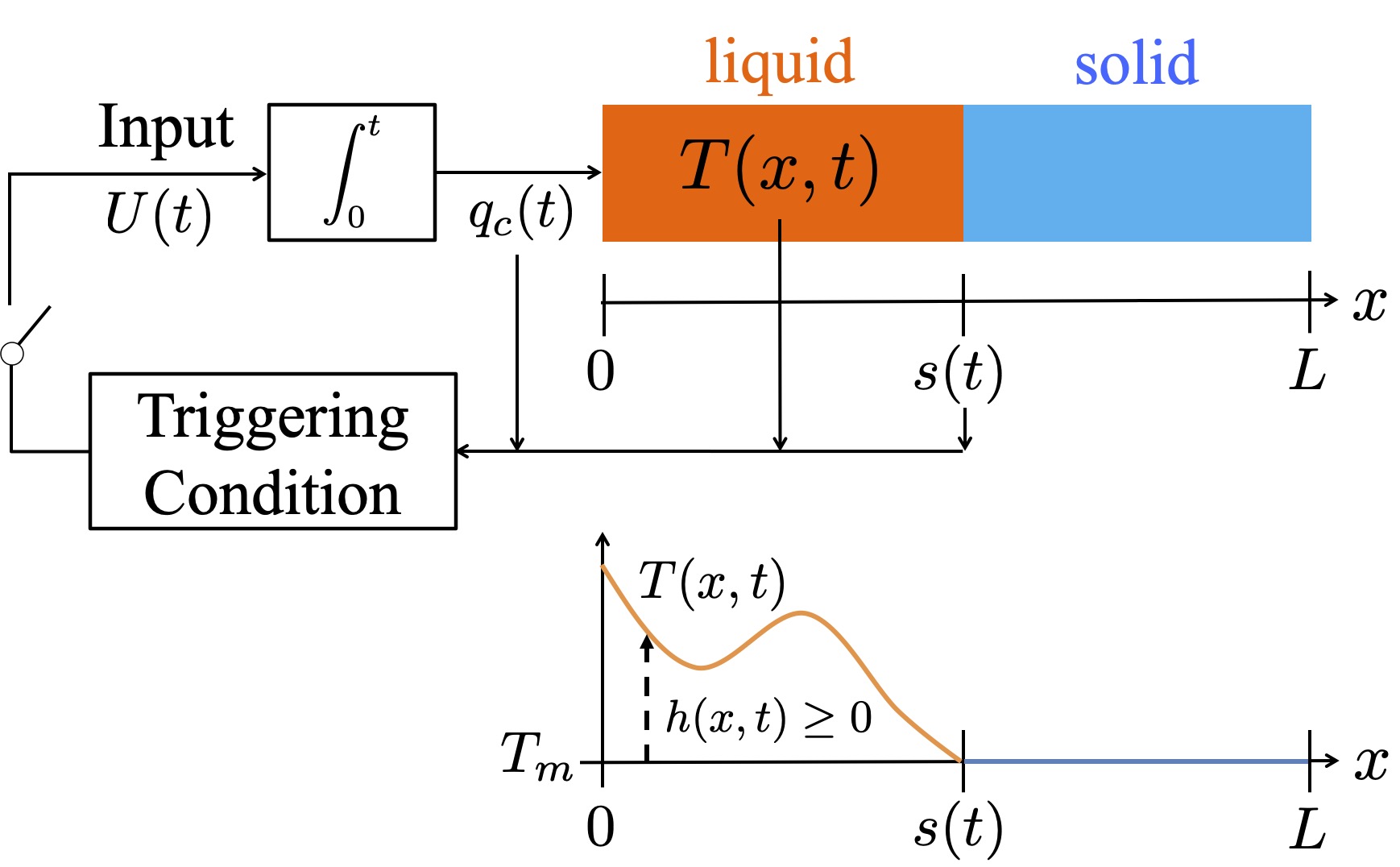}\\
\caption{Schematic of the one-phase Stefan problem with actuator dynamics under event-triggered control.}
\label{fig:stefan}
\end{figure}

Consider the melting or solidification  in a material of length $L$ in one dimension (see  Fig.~\ref{fig:stefan}). Divide the domain $[0, L]$ into two time-varying sub-intervals: $[0,s(t)]$, which contains the liquid phase, and $[s(t),L]$, that contains the solid  phase. Let the heat flux be entering at the boundary of the liquid phase to promote the melting process

The energy conservation and heat conduction laws yield the heat equation of  the liquid phase, the boundary conditions, the dynamics of the moving boundary, and the initial values as follows:

\begin{align}\label{eq:stefanPDE}
T_t(x,t)&=\alpha T_{xx}(x,t), \hspace{2mm}  \textrm{for} \hspace{2mm} t > 0, \hspace{2mm} 0< x< s(t), \\ 
\label{eq:stefancontrol}
-k T_x(0,t)&=q_{\rm c}(t),  \hspace{2mm} \textrm{for} \hspace{2mm} t >0,\\ \label{eq:stefanBC}
T(s(t),t)&=T_{{\rm m}}, \hspace{2mm} \textrm{for} \hspace{2mm} t >0, \\
 \label{eq:stefanODE}
\dot s(t) & =- \beta T_x(s(t),t) , \\
\label{eq:stefanIC}
 s(0) &=  s_0, \textrm{and } T(x,0) = T_0(x),  \textrm{for } x \in (0, s_0]. 
\end{align}
The heat flux $q_{\rm c}(t)$ is manipulated by the voltage input $U(t)$ modeled by a first-order actuator dynamics:
\begin{align} \label{eq:stefan-actuator} 
    \dot q_{c}(t) =  U(t).  
\end{align} 

There are two requirements for the validity of the model
\begin{align}\label{temp-valid}
T(x,t) \geq& T_{{\rm m}}, \quad  \forall x\in(0,s(t)), \quad \forall t>0, \\
\label{int-valid}0 < s(t)<  &L, \quad \forall t>0. 
\end{align}

First, the trivial: the liquid phase is not frozen, i.e., the liquid temperature $T(x,t)$ is greater than the melting temperature $T_{\rm m}$. Second, equally trivially, the material is not entirely in one phase, i.e., the interface remains inside the material's domain. These physical conditions are also required for the existence and uniqueness of  solutions \cite{Gupta03}.
Hence, we assume the following for the initial data. 

\begin{assum}\label{ass:initial} 
$0 < s_0 < L$, $T_0(x) \in C^1([0, s_0];[T_{\rm m}, +\infty))$ with $T_0(s_0) = T_{\rm m}$.
 \end{assum}

We remark the following lemma. 

\begin{lemma}\label{lem1}
With Assumption \ref{ass:initial}, if $q_{\rm c}(t)$ is a bounded piecewise continuous non-negative heat function, i.e.,
\begin{align}
q_{{\rm c}}(t) \geq 0,  \quad \forall t\geq 0,  
\end{align}
then there exists a unique classical solution for the Stefan problem \eqref{eq:stefanPDE}--\eqref{eq:stefanODE}, which satisfies \eqref{temp-valid}, and 
\begin{align} \label{eq:sdot-pos} 
    \dot s(t) \geq 0, \quad \forall t \geq 0. 
\end{align}

\end{lemma}

The definition of the classical solution of the Stefan problem is given in 
Appendix A of \cite{Shumon19journal}. The proof of Lemma \ref{lem1} is by maximum principle for parabolic PDEs and Hopf's lemma, as shown in \cite{Gupta03}. 

Control Barrier Functions (CBFs) are introduced to render the equivalency between the forward invariance of a safe set satisfying the imposed constraints with the positivity of the CBFs. Followng \cite{koga2022ACC}, let $h_1(t)$, $h_2(t)$, $h_3(t)$, and $h(x,t)$ be CBFs defined by 
\begin{align}
    h_1(t) :&= \sigma(t)  \nonumber\\
    & =  -\left( \frac{k}{\alpha} \int_0^{s(t)} (T(x,t) - T_m) dx + \frac{k}{\beta} (s(t) - s_r) \right), \label{eq:h1_def} \\
    h_2(t) &= q_{\rm c}(t), \label{eq:h2_def}\\
    h_3(t) &= - q_{\rm c}(t) + c_1 \sigma(t), \label{eq:h3_def}\\
    h(x,t) &= T(x,t) - T_{\rm m}. \label{eq:h_def}
\end{align}

The functions $h_1(t)$ and $h_2(t)$ defined above can be seen as valid CBFs to satisfy the state constraints in \eqref{temp-valid} and \eqref{int-valid} as shown in the following lemma.  

\begin{lemma} \label{lem:overshoot} 
    With Assumption \ref{ass:initial}, suppose that the following conditions hold: \begin{align}
     \label{eq:h1-pos}   h_1(t) \geq 0, \\
     \label{eq:h2-pos}     h_2(t) \geq 0,  
    \end{align}
    for all $t \geq 0$. Then, it holds that 
    \begin{align} \label{eq:hxt-pos} 
    h(x,t) \geq & 0, \quad \forall x \in (0, s(t)), \quad \forall t \geq 0, \\
      0< s_0 \leq   s(t) \leq & s_r, \quad \forall t \geq 0,  \label{eq:s-ineq} 
    \end{align}
    under the classical solution of \eqref{eq:stefanPDE}--\eqref{eq:stefanODE}.
    \end{lemma}

Lemma \ref{lem:overshoot} is proven with Lemma \ref{lem1}. To validate the conditions \eqref{eq:h1-pos} and \eqref{eq:h2-pos} for all $t \geq 0$, at least the conditions must hold at $t = 0$, which necessitate the following assumptions on the initial condition and the setpoint restriction. 

\begin{assum}
    \label{ass:h2} 
$0 \leq q_{\rm c}(0) $. 
\end{assum} 

\begin{assum}
    \label{ass:h1} 
The setpoint position $s_{\rm r}$ is chosen to satisfy 
\begin{align}
s_{0} +    \frac{\beta}{\alpha} \int_0^{s_0} (T_0(x) - T_{\rm m}) dx  \leq s_{\rm r} < L. 
\end{align}
\end{assum}

Note that CBF $h_3(t)$ defined in \eqref{eq:h3_def} satisfies 
\begin{align} \label{eq:doth1-h3}
\dot h_1(t) = - c_1 h_1(t) + h_3(t).
\end{align} 
This function is introduced to handle the property that the relative degree of $h_1(t)$ is more than 1. We can see it by $\dot h_1 = - q_{\rm c}(t) = - h_2(t)$, which is not explicitly dependent on the control input $U(t)$. However, as presented in the next section, the time derivative of $h_3(t)$ is explicitly dependent on $U(t)$. Moreover, once the positivity of $h_3(t)$ is ensured, then the positivity of $h_1(t)$ is also guaranteed owing to \eqref{eq:doth1-h3}. Thus, in later sections, for ensuring the safety, we mainly focus on guaranteeing the positivity of $h_2(t)$ and $h_3(t)$. 

\section{Event-Triggered Control Design}

In this section, we propose an event-triggered controller for both safety and stability purposes. \cite{koga2022ACC} develops the nonovershooting control
\begin{align} \label{eq:nonover} 
    U^*(t) &= - c_1 h_2(t) + c_2 h_3(t) \notag\\
    &=  - (c_1 + c_2) q_{\rm c}(t) + c_1 c_2 \sigma (t) ,  
\end{align}
under the continuous-time design, where $c_1>0$ and $c_2 >0$ are positive constants satisfying
\begin{align} \label{eq:c1_condition} 
     c_1 \geq \frac{q_{\rm c}(0)}{\sigma(0)} . 
\end{align}
The closed-loop system with $U(t) = U^*(t)$ is ensured to satisfy the constraints \eqref{temp-valid} and \eqref{int-valid} by showing the positivity of all CBFs, and also shown to be globally exponentially stable. 

In this paper, we consider a digital control by applying Zero-Order Hold (ZOH) to the continuous-time control law \eqref{eq:nonover}, given as 
\begin{align} \label{eq:digital}
    U(t) \equiv U^*(t_j), \quad \forall t \in [t_j, t_{j+1}), 
\end{align}
and derive a triggering condition for $t_j$ to maintain both the safety and stability. 

\subsection{Triggering Condition for Safety}

Under \eqref{eq:digital} with \eqref{eq:nonover}, we design the triggering condition so that the following inequalities are satisfied: 
\begin{align}
\label{eq:ddth2} \dot h_2(t) & \geq - \bar c_1 h_2(t) , \\
    \label{eq:ddth3} \dot h_3(t) & \geq - \bar c_2 h_3(t),  
\end{align}
for some positive constants $\bar c_i>0$. The triggering condition is obtained by using the following procedure. 

First, we start with taking time derivative of \eqref{eq:h1_def} and \eqref{eq:h2_def} and using \eqref{eq:stefan-actuator} so, we have
\begin{align}
\label{eq:dot2h}
    \dot{h}_2(t)
    & =U(t), \\
     \dot{h}_3(t) & =-U(t)-c_1 h_2(t) 
  \label{eq:h3dot}
\end{align}
Under the digital control  \eqref{eq:digital}, by defining \begin{align}
    \tilde U^*(t) := U^*(t) - U^*(t_j), 
\end{align}
\eqref{eq:dot2h} and \eqref{eq:h3dot} yield
\begin{align}
\label{eq:dot2h1}
    \dot{h}_2(t)
    & = - c_1 h_2(t) + c_2 h_3(t) 
 - \tilde U^*(t), \\
     \dot{h}_3(t) & = - c_2 h_3(t) + \tilde U^*(t). 
  \label{eq:h3dot1}
\end{align}
Thus, one can see that to ensure \eqref{eq:ddth2} and \eqref{eq:ddth3}, it must hold 
\begin{align} 
(c_2-\bar c_2)h_3(t) \leq    \tilde{U}^*(t) \leq (\bar c_1 - c_1) h_2(t) +c_2h_3(t) 
   \label{eq:Uhat_up}
\end{align}

The inequality \eqref{eq:Uhat_up} must be satisfied at least $t = t_j$. Noting that $\tilde U^*(t_j) = 0$ and $h_2(t_j) \geq 0$ and $h_3(t_j) \geq 0$, the parameters $\bar c_1$ and $\bar c_2$ introduced in \eqref{eq:ddth2} \eqref{eq:ddth3} can be set as 
\begin{align}
    \bar c_1&=(1+\delta_1) c_1 \\
    \bar c_2&=(1+\delta_2)c_2 
\end{align}
where $\delta_2\geq0$ and $\delta_1\geq 0$ are gain parameters. Then, the condition \eqref{eq:Uhat_up} is rewritten as 
\begin{align}
 - \delta_2 c_2  h_3(t) \leq    \tilde{U}^*(t) \leq \delta_1 c_1 h_2(t) +c_2h_3(t)  
   \label{eq:cond-safe}
\end{align}



Then, we can define the set of event times $I={t_0,t_1,t_2,...}$ with $t_0=0$ forms an increasing sequence with the following rule by using \eqref{eq:cond-safe}
\begin{align}
    t_{j+1}=\inf\{\mathcal{S}(t,t_j)\}
\end{align}
where
\begin{align} 
    &\mathcal{S}(t,t_j)=\left\{t\in \mathbb{R}_+|\left(t>t_j\right) \wedge \left( \left( -\delta_2 c_2 h_3(t)> \tilde U^*(t) \right) \right. \right. \nonumber \\ &\left. \left. \vee \left( \tilde U^*(t)> \delta_1 c_1 h_2(t) +c_2 h_3(t)\right) \right) \right\} . \label{eq:event-trigger-safe}
\end{align}


\subsection{Triggering Condition for Stability}

The event-triggering mechanism for stability can be derived by considering the backstepping approach as follows. Let $X(t)$ be reference error variable defined by $X(t):= s(t)-s_{r}$. Then, the system \eqref{eq:stefanPDE}--\eqref{eq:stefanODE} is rewritten with respect to $h(x,t)$ defined in \eqref{eq:h_def}, $h_2(t)$ defined in \eqref{eq:h2_def}, and $X(t)$ as
\begin{align}\label{u-sys1}
h_{t}(x,t) &=\alpha h_{xx}(x,t),\\
\label{u-sys2}h_x(0,t) &= - h_2(t)/k,\\
\dot h_2(t) &= U(t), \\
\label{u-sys3}h(s(t),t) &=0,\\
\label{u-sys4}\dot X(t) &=-\beta h_x(s(t),t). 
\end{align}


Following Section 3.3. in \cite{koga2021towards}, we introduce the following forward and inverse transformations:  
\begin{align}\label{eq:DBST}
w(x,t)&=h(x,t)-\frac{\beta}{\alpha} \int_{x}^{s(t)} \phi (x-y)h(y,t) {\rm d}y \notag\\
&-\phi(x-s(t)) X(t), \\
\label{kernel}\phi(x) &= c_1 \beta^{-1} x- \varepsilon, \\
\label{inv-trans}
h(x,t)&=w(x,t)-\frac{\beta}{\alpha} \int_{x}^{s(t)} \psi (x-y)w(y,t) {\rm d}y \notag\\
&-\psi(x-s(t)) X(t), \\
\label{inv-gain}
\psi(x) &= e^{ \bar \lambda x } \left( p_1 \sin\left( \omega x \right) + \varepsilon \cos\left( \omega x \right) \right) , 
\end{align}
where $\bar \lambda = \frac{\beta \varepsilon}{2 \alpha}$, $\omega = \sqrt{\frac{4 \alpha c_1 - (\varepsilon\beta)^2 }{4 \alpha^2 } }$, $p_1 =  - \frac{1}{2 \alpha \beta \omega} \left( 2 \alpha c_1 - (\varepsilon \beta )^2 \right) $, and $0<\varepsilon<2 \frac{\sqrt{\alpha c_1}}{\beta}$ is to be chosen later. 
As derived in Section 3.3. in \cite{koga2021towards}, taking the spatial and time derivatives of \eqref{eq:DBST} along the solution of \eqref{u-sys1}-\eqref{u-sys4}, and noting the CBF $h_3(t)$ defined in \eqref{eq:h3_def} satisfying \eqref{eq:ddth3}, one can obtain the following target system: 
\begin{align}\label{tarPDE}
w_t(x,t)&=\alpha w_{xx}(x,t)+ \dot{s}(t) \phi'(x-s(t))X(t) 
, \\
\label{tarBC2} w_x(0,t) &= \frac{h_3(t)}{ k} - \frac{\beta}{\alpha}\varepsilon \bigg[  w(0,t)-\frac{\beta}{\alpha} \int_{0}^{s(t)} \psi (-y)w(y,t) {\rm d}y \notag\\
& -\psi(-s(t)) X(t) \bigg].  \\
\dot h_3(t) &= - c_2 h_3(t) + \tilde U^*(t), \label{tar_h3} \\
\label{tarBC1} w(s(t),t) &= \varepsilon X(t), \\
\label{tarODE}\dot X(t)&=-c_1 X(t)-\beta w_x(s(t),t).
\end{align}

Note that \eqref{tarBC2} is derived using $ h_3(t) = -q_{\rm c}(t) + c_1 \sigma(t) = - h_2(t) - c_1 \left( \frac{k}{\alpha} \int_0^{s(t)} (T(x,t) - T_{\rm m}) dx + \frac{k}{\beta} (s(t) - s_{\rm r}) \right) \notag\\
    = - h_2(t) - k \left( \frac{\beta}{\alpha} \int_0^{s(t)} \phi'(-y) h(y,t) dy + \phi'(-s(t)) X(t) \right)$, with the helpf of \eqref{kernel}. The objective of the transformation \eqref{eq:DBST} is to add a stabilizing term $-c_1 X(t)$ in \eqref{tarODE} of the target $(w,X)$-system which is easier to prove the stability than $(u,X)$-system. 

    We derive the triggering condition to satisfy the following inequality 
    \begin{align}
    | \tilde U^*(t)| \leq  (1 - \delta_2) c_2 h_2(t) + \delta_2 c_2 h_3(t) , \hspace{3mm} \delta_2 \in (0,1),  \label{eq:cond-stability}
    \end{align}
    which ensures the stability by Lyapunov analysis as shown in the next section. 
    Combining the condition \eqref{eq:cond-stability} for stability with the condition \eqref{eq:cond-safe} for safety, the resulting required condition for simultaneous safety and stability is obtained as 
    \begin{align}
 - \delta_2 c_2  h_3(t) & \leq    \tilde{U}^*(t) \leq \mu_1 h_2(t) + \delta_2 c_2 h_3(t)  
   \label{eq:cond-safe-stable}, \\
   \mu_1 &= \min\{ \delta_1 c_1, (1 - \delta_2) c_2 \} \label{eq:mu1-def} 
\end{align}
    thereby the event-triggering mechanism is described by 
\begin{align}
t_{j+1} &= \inf \left\{t\in \mathbb{R}_+|\left(t>t_j\right) \wedge \left( \left( -\delta_2 c_2 h_3(t)> \tilde U^*(t) \right) \right. \right. \nonumber \\ &\left. \left. \vee \left( \tilde U^*(t)> \mu_1 h_2(t) + \delta_2 c_2 h_3(t)\right) \right) \right\} . \label{eq:triggering-safe-stable}
\end{align}

    \section{Closed-Loop Analysis and Main Results}

This section provides the theoretical analysis of the closed-loop system under the proposed event-triggered control law. 

\subsection{Avoidance of Zeno Behavior}

One potential issue of the event-triggered control is the so-called ``Zeno" behavior, which causes infinite triggering times within finite time interval. Such behavior essentially does not enable the digital control implementation. The Zeno behavior can be proven not to exist by showing the existence of the minimum dwell-time. We state the following theorem. 

\begin{theorem}
Let Assumptions \ref{ass:initial}--\ref{ass:h1} hold. Consider the closed-loop system consisting of the plant \eqref{eq:stefanPDE}--\eqref{eq:stefan-actuator} and the event-triggered boundary control \eqref{eq:nonover} with the gain condition \eqref{eq:c1_condition} and the triggering mechanism \eqref{eq:triggering-safe-stable}. There exists a minimal dwell-time between two
triggering times, i.e. there exists a constant $\tau >0$ (independent of the initial condition) such that $t_{j+1} - t_j \geq \tau$, for
all $j \geq 0$. Moreover, all CBFs defined as \eqref{eq:h1_def}--\eqref{eq:h_def} satisfy the constraints $h_1(t) \geq 0$, $h_2(t) \geq 0$ for all $t \geq 0$, and $h(x,t) \geq 0$ for all $x \in (0, s(t))$ and for all $ t \geq 0$. 
\end{theorem}

\begin{proof}

Let $m_1(t)$ and $m_2(t)$ be defined by
\begin{align}
        m_1(t)
    &= \mu_1 h_2(t) + \delta_2 c_2 h_3(t)-\tilde{U}^*(t)\nonumber \\
    &= ( \mu_1 + c_1 ) h_2(t) - (1 -  \delta_2) c_2 h_3(t) + U^*(t_j),\\
m_2(t)
    &=\tilde{U}^*(t) + \delta_2 c_2 h_3(t)\nonumber \\
    &=- U^*(t_j)-c_1h_2(t)+\bar c_2 h_3(t).
\end{align}
Since the event-triggering mechanism ensures both $m_1(t) >0$ and $m_2(t)>0$ for all $t \in [t_{j}, t_{j+1})$ and either $m_1(t_{j+1}) = 0$ or $m_2(t_{j+1}) = 0$ holds, we show that there exists a positive constant $\tau>0$ (lower bound of the dwell-time), which is independent on triggered time $t_j$, such that both $m_1(\bar t) \geq 0$ and $m_2( \bar t) \geq 0$ hold for some $ \bar t - t_j \geq \tau$.  
By taking the time derivatives of $m_1(t)$ and $m_2(t)$, we get
 \begin{align} 
\dot m_1(t) &=  ( \mu_1 + c_1 + (1 -  \delta_2) c_2 ) U^*(t_j) + (1 -  \delta_2) c_1 c_2 h_2(t)\label{eq:dotm1} \\
\ddot m_1(t) &=  (1 -  \delta_2) c_1 c_2 U^*(t_j)\label{eq:ddotm1} \\
    \dot{m}_2(t)&=-\left( c_1 + \bar c_2\right)U^*(t_j)-\bar c_2 c_1 h_2(t),\\
    \ddot{m}_2(t)&=-c_1\bar c_2U^*(t_j).\label{eq:ddotm2}
\end{align}
Since \eqref{eq:ddotm1} and \eqref{eq:ddotm2} are constant in time, explicit solutions for $m_1(t)$ and $m_2(t)$ with respect to time $t \geq t_{j}$ are obtained as follows. 
\begin{align}
m_1(t)=& \frac{(1- \delta_2) c_1 c_2 }{2}U^*(t_j)(t-t_j)^2 \notag\\
& + \dot m_1(t_j) (t - t_j) + m_1(t_j)\nonumber \\
    =& - \frac{(1- \delta_2) c_1 c_2 }{2} ( c_1 h_2(t_j) - c_2 h_3(t_j))(t-t_j)^2 \notag\\
    & - \left( c_1 (\mu_1 + c_1) h_2(t_j) \right.  \notag\\
    & \quad \left.- c_2 (c_1 + c_2(1 - \delta_2)) h_3(t_j) \right)(t - t_j) \notag\\
    & + \mu_1 h_2(t_j) + \delta_2 c_2 h_3(t_j) 
    \label{eq:m1-sol}
 \\
    m_2(t)=& - \frac{c_1\bar c_2 U^*(t_j)}{2}(t-t_j)^2+\dot{m}_2(t_j)(t-t_j)+m_2(t_j)\nonumber \\
    =
    & \frac{c_1\bar c_2 ( c_1 h_2(t_j) - c_2 h_3(t_j))}{2}(t-t_j)^2\nonumber \\
    &+ \left(c_1^2 h_2(t_j) -(c_1+\bar c_2)  c_2 h_3(t_j) \right)(t-t_j)\nonumber \\
    & +  \delta_2 c_2 h_3(t_j) . \label{eq:m2-sol}
\end{align}
Owing to the positivity of CBFs $h_2(t)$ and $h_3(t)$ ensured by \eqref{eq:ddth2} and \eqref{eq:ddth3} under the event-triggered mechanism \eqref{eq:event-trigger-safe}, the solutions \eqref{eq:m1-sol} \eqref{eq:m2-sol} for $t \geq t_j$ satisfy the following inequalities:  
\begin{align} \label{eq:m1-sol-ineq}
m_1(t) \geq & h_2(t_j) \left( - \frac{(1- \delta_2) c_1^2 c_2 }{2}  (t-t_j)^2 
\right.  \notag\\ &\left.  
    -  c_1 (\mu_1 + c_1) (t - t_j)  + \mu_1 \right)  , \\
m_2(t) \geq 
    &c_2 h_3(t_j) \left( - \frac{c_1\bar c_2 }{2}(t-t_j)^2
     - (c_1+\bar c_2)  (t-t_j) +  \delta_2 \right). \label{eq:m2-sol-ineq}
\end{align}
One can see that the right hand sides of both \eqref{eq:m1-sol-ineq} and \eqref{eq:m2-sol-ineq} maintain nonnegative for all $t \in [t_j, t_j + \tau)$, where 
\begin{align}
\tau & =  \min \left\{ \tau_1, \tau_2 
\right\} , \\
\tau_1 & = \frac{ - c_1(\mu_1 + c_1) + \sqrt{c_1^2 (\mu_1 + c_1)^2 + 2 \mu_1(1 - \delta_2) c_1^2 c_2 }}{(1 - \delta_2) c_1^2 c_2}, \\
\tau_2 & = \frac{ - (c_1 + \bar c_2) + \sqrt{(c_1 + \bar c_2)^2 + 2 \delta_2  c_1 \bar c_2}}{c_1 \bar c_2}, 
\end{align}
which stands as the minimum dwell-time. 


\end{proof}

\subsection{Stability Analysis}

We prove the stability of the closed-loop system as presented below. 

\begin{theorem} \label{thm:nonover}
Let Assumptions \ref{ass:initial}--\ref{ass:h1} hold. Consider the closed-loop system consisting of the plant \eqref{eq:stefanPDE}--\eqref{eq:stefan-actuator} and the event-triggered boundary control \eqref{eq:nonover} with the gain condition \eqref{eq:c1_condition} and the triggering mechanism \eqref{eq:triggering-safe-stable}. The closed-loop system is  exponentially stable at the equilibrium $s=s_{\rm r}, T(x,\cdot) \equiv T_{\rm m}, q_{\rm c} = 0$,  in the sense of the following norm: 
\begin{align} \label{eq:Phi-def} 
\Phi(t):= || T[t] - T_{\rm m} ||^2 +(s(t) - s_{\rm r})^2 + q_{\rm c}(t)^2,  
\end{align}
for all initial conditions in the safe set, i.e., globally. In other words, there exist positive constants $M>0$ and $b>0$ such that the following norm estimates hold: 
\begin{align}
\Phi(t) \leq M \Phi(0) e^{- bt} .  
\end{align}
\end{theorem}
\begin{proof}
What remains to show is the stability of the system by Lyapunov analysis. Owing to the equivalent stability property through the backstepping transformation, we employ Lyapunov analysis to the target system \eqref{tarPDE}--\eqref{tarODE}. Following Lemma 20 in \cite{koga2021towards}, by introducing a Lyapunov function $V(t)$ defined by 
\begin{align}\label{app:lyap}
V(t) = \fr{1}{2\alpha } || w[t]||^2 + \fr{\varepsilon}{2\beta } X(t)^2,
\end{align}
one can see that there exists a positive constant $\ep^*>0$ such that for all $\ep \in (0, \ep^*)$ the following inequality holds: 
\begin{align}\label{app:dotV2}
\dot V(t) \leq &  - b V (t)    + \fr{ 2 s_{\rm r}}{k^2} h_3(t)^2+ a \dot{s}(t) V(t) , 
\end{align}
where $a = \fr{2\beta \ep }{\alpha} \max \left\{1,\fr{\alpha c^2 s_{{\rm r}}}{2\beta^3 \ep^3} \right\}$, $b =\fr{1}{8} \min \left\{ \fr{\alpha}{s_{{\rm r}}^2}, c \right\}$, and the condition $\dot s(t) \geq 0$ ensured in Lemma \ref{lem1} is applied. We further introduce another Lyapunov function of $h_3$, defined by 
\begin{align} \label{eq:Vh-def} 
    V_h(t) = \fr{1}{2} h_3(t)^2 + \frac{q}{2} h_1(t)^2,  
\end{align}
with a positive constant $q>0$. 
Taking the time derivative of \eqref{eq:Vh-def} and applying \eqref{eq:doth1-h3} and \eqref{tar_h3} yields 
\begin{align} \label{eq:Vh-der} 
    \dot V_h(t) = &   - c_2 h_3(t)^2 + h_3(t) \tilde U^*(t) \notag\\
    & + q \left( - c_1 h_1(t)^2 + h_1(t) h_3(t)\right). 
\end{align}
Under the event-triggered mechanism, the inequality \eqref{eq:cond-stability} is satisfied, thereby \eqref{eq:Vh-der} leads to 
\begin{align} \label{eq:Vh-der-ineq} 
    \dot V_h(t)  \leq & -  c_2 (1 - \delta_2) h_3(t)^2 + \mu_1 h_2 (t)h_3(t)  \notag\\
    & - q c_1 h_1(t)^2 + q  h_1(t) h_3(t). 
\end{align}
From the definition \eqref{eq:h3_def}, we have $h_2 = - h_3 + c_1 h_1$. Substituting this into \eqref{eq:Vh-der-ineq}, applying Young's inequality to the cross term $h_1 h_3$, one can derive the following inequality: 
\begin{align}
\label{eq:Vh-der-ineq2} 
    \dot V_h(t)  \leq & - \left(  c_2 (1 - \delta_2) + \frac{\mu_1}{2} - \frac{q}{2 c_1} \right) h_3(t)^2  \notag\\
    & - \frac{c_1}{2} \left( q - c_1 \mu_1 \right) h_1(t)^2. 
\end{align}
Let $\bar V$ be the Lyapunov function defined by 
\begin{align} \label{eq:Vbar-def} 
    \bar V = V + p V_h.  
\end{align}
Applying \eqref{app:dotV2} and \eqref{eq:Vh-der-ineq2} with setting $q = 2 c_1 \mu_1$, $p = \fr{ 8 s_{\rm r}}{c_2 k^2 (1 - \delta_2) }$, and $\mu_1 \leq (1 - \delta_2) c_2$ by \eqref{eq:mu1-def}, the time derivative of \eqref{eq:Vbar-def} is shown to satisfy 
\begin{align} \label{eq:Vbar-der} 
    \dot {\bar V}(t) \leq - \bar b \bar V(t) + a \dot{s}(t) \bar V(t),  
\end{align}
where $\bar b = \min\{b, 2 s_{\rm r} / k^2\}$. 
As performed in \cite{koga2021towards}, with the condition $0 < s(t) \leq s_{\rm r}$ ensured in Lemma \ref{lem:overshoot}, the differential inequality \eqref{eq:Vbar-der} leads to 
\begin{align}
    \bar V(t) \leq e^{a s_{\rm r}} \bar V(0) e^{- \bar b t}, 
\end{align}
which ensures the exponential stability of the target system \eqref{tarPDE}--\eqref{tarODE}. Due to the invertibility of the transformation \eqref{eq:DBST} \eqref{inv-trans}, one can show the exponential stability of the closed-loop system, which completes the proof of Theorem \ref{thm:nonover}.

\end{proof}
\section{Simulations}
We perform the numerical simulation considering a strip of zinc, the physical parameters of which are given in Table 1 in \cite{Shumon19journal}. The initial interface position is set to $s_0=5 \ [cm]$, and a linear profile is considered for the initial temperature profile which is $T_0(x)=\bar T (1 - x/ s_0) + T_{\rm m}$ with $\bar T=1 \ [^\circ C]$ and $T_m = 420 \ [^\circ C]$. The setpoint position is set as $30 \ [cm]$, and control gains are considered as $\delta_1=10$, $c_1=3.2\times10^{-3}$ and $c_2=0.5\times10^{-2}$. Fig.~\ref{fig:Ut} depicts the event-based control input under two choices of gain parameters $\delta_2=0.3$ and $\delta_2=0.7$. The figure illustrates that the input only requires a few number of control updates to ensure stability and safety of the system as opposed to the continuous-time input signal. Fig.~\ref{fig:qc} demonstrates that the boundary temperature is increased first, then cooled down once the actuator added enough heat and remained positive during this process. In Fig~\ref{fig:st}, the interface position converges to the setpoint without any overshooting with the proposed control law. Fig.~\ref{fig:h1}, is the imposed CBF for and energy amount, stays positive and satisfies CBF condition. In addition, temperature evolution along the domain is presented in Fig.~\ref{fig:hxt}. Starting a linear initial temperature profile, the liquid temperature successfully maintains above the melting temperature with event-triggered control law as the safety of the system. Fig.~\ref{fig:Util} shows the time-evolution of the control signal with the triggering condition \eqref{eq:cond-safe-stable}. Once the input error $\tilde U^*(t)$ reaches the lower bound $-\delta_2c_2h_3(t)$ or upper bound $\min(\delta_1c_1,(1-\delta_2)c_2)q_c(t)+\delta_2c_2h_3(t)$, the triggering event is caused by closing the loop. Then, the input error signal is reset to zero due to the update of the control signal. As a time-evolution, it can be seen that both the lower and upper bounds converge to zero, and the number of events decreases when liquid temperature, $T(x,t)$, converges to the melting temperature, $T_m$.   

\begin{figure}[t]
\centering
\includegraphics[width=0.99\linewidth]{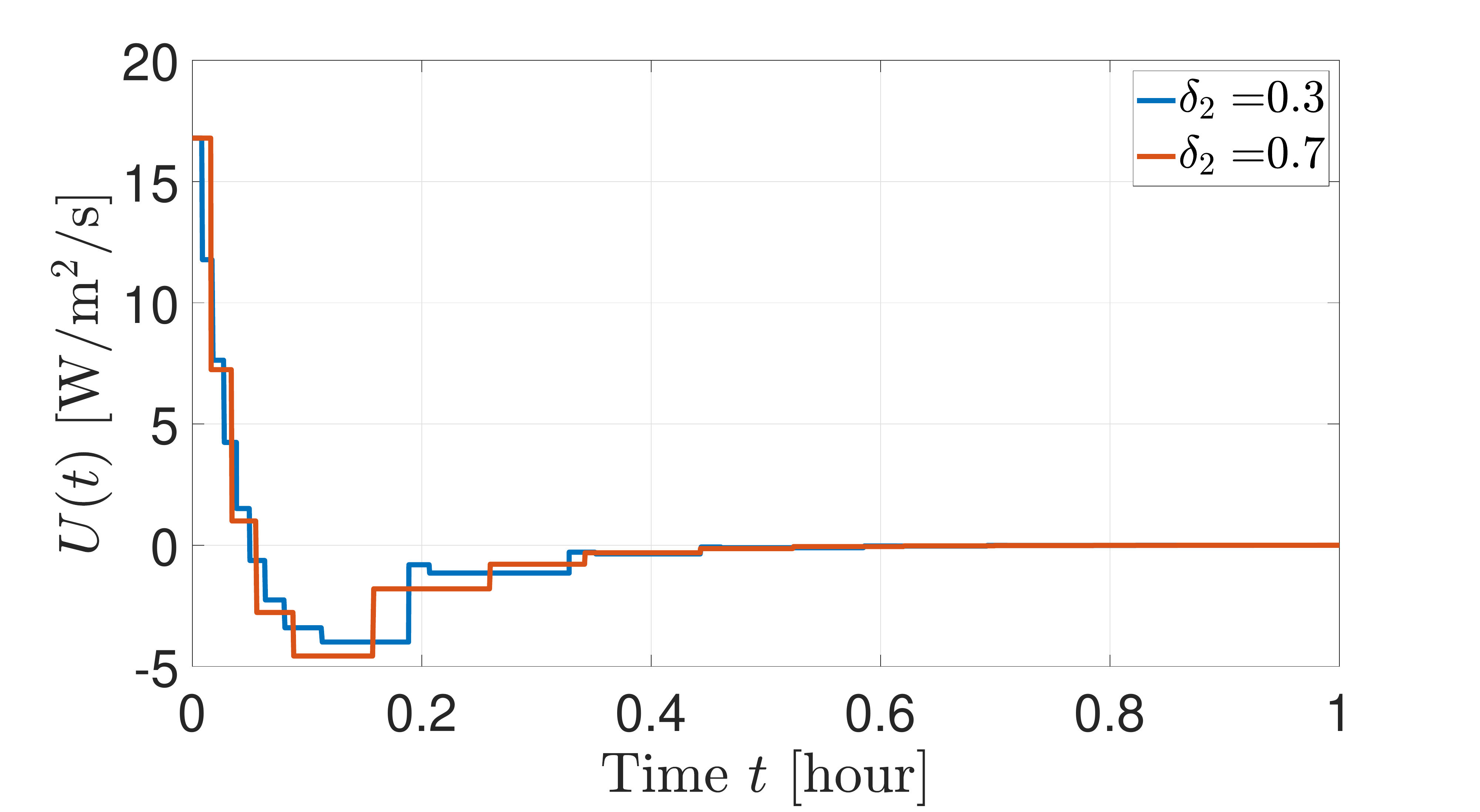}\\
\caption{Event-based control input under $\delta_2 = 0.3$ and $\delta_2 = 0.7$.}
\label{fig:Ut}
\end{figure}

\begin{figure}[t]
\centering
\includegraphics[width=0.99\linewidth]{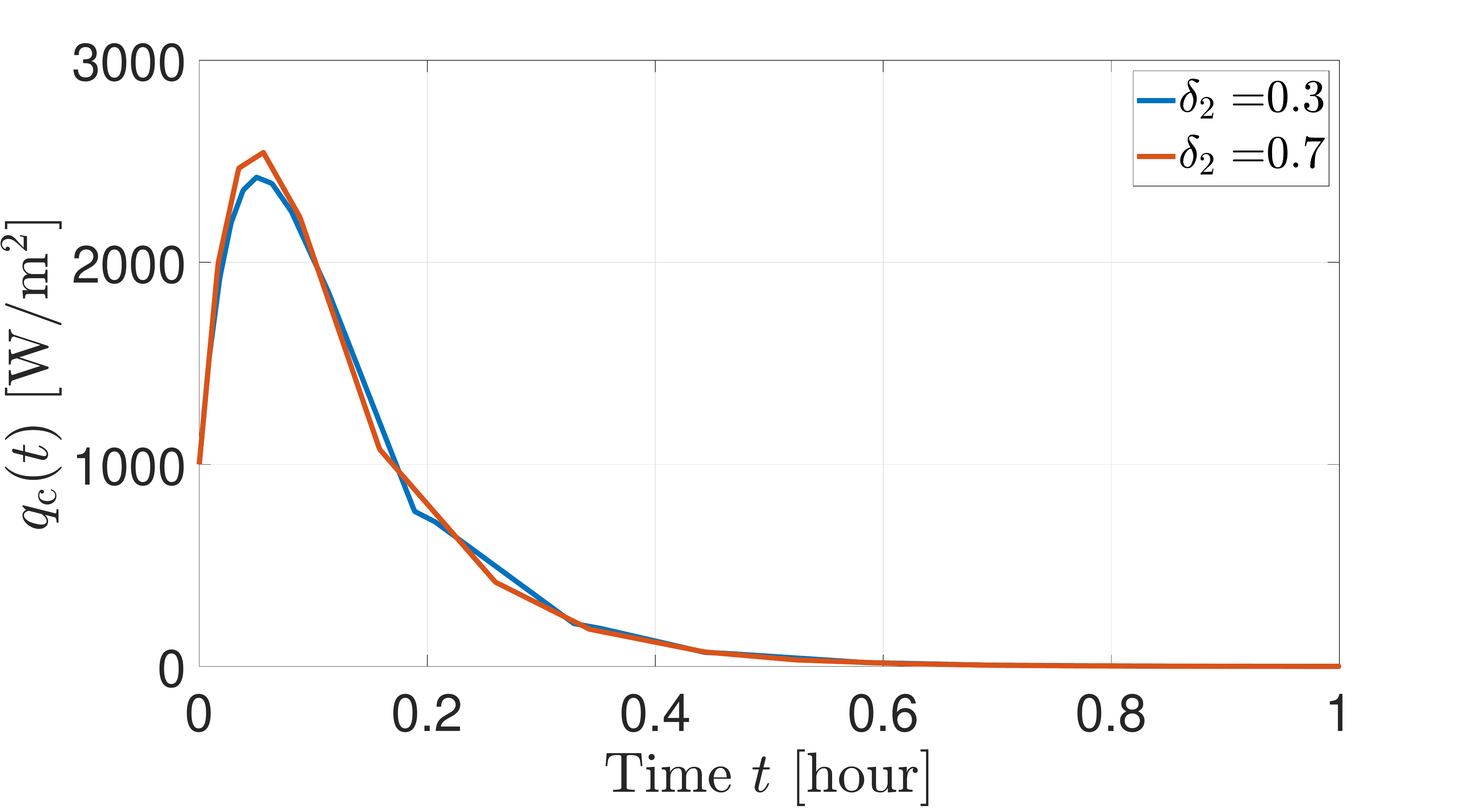}\\
\caption{Boundary heat flux remains positive.}
\label{fig:qc}
\end{figure}

\begin{figure}[t]
\centering
\includegraphics[width=0.99\linewidth]{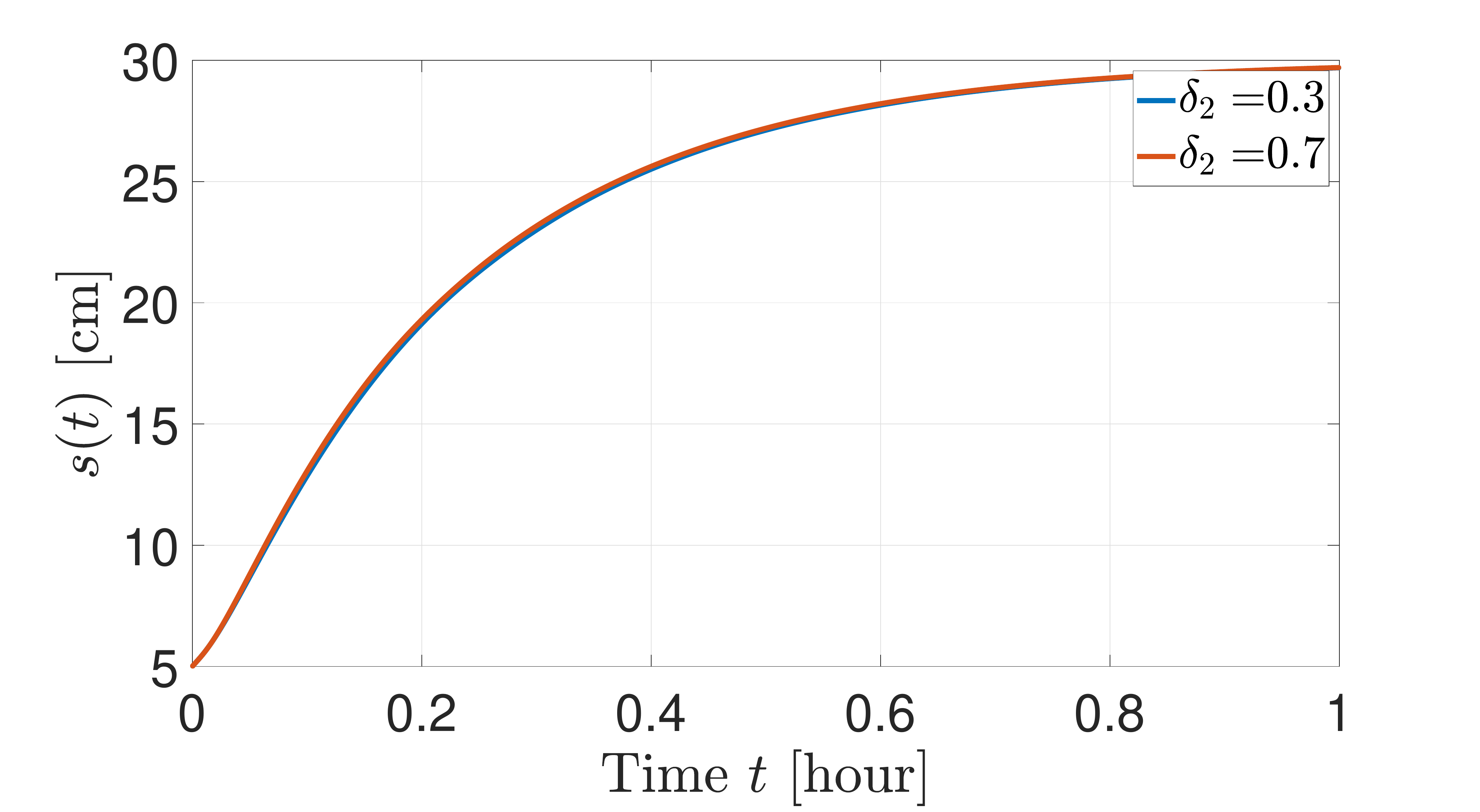}\\
\caption{Interface position converges to setpoint without overshooting. }
\label{fig:st}
\end{figure}

\begin{figure}[t]
\centering
\includegraphics[width=0.99\linewidth]{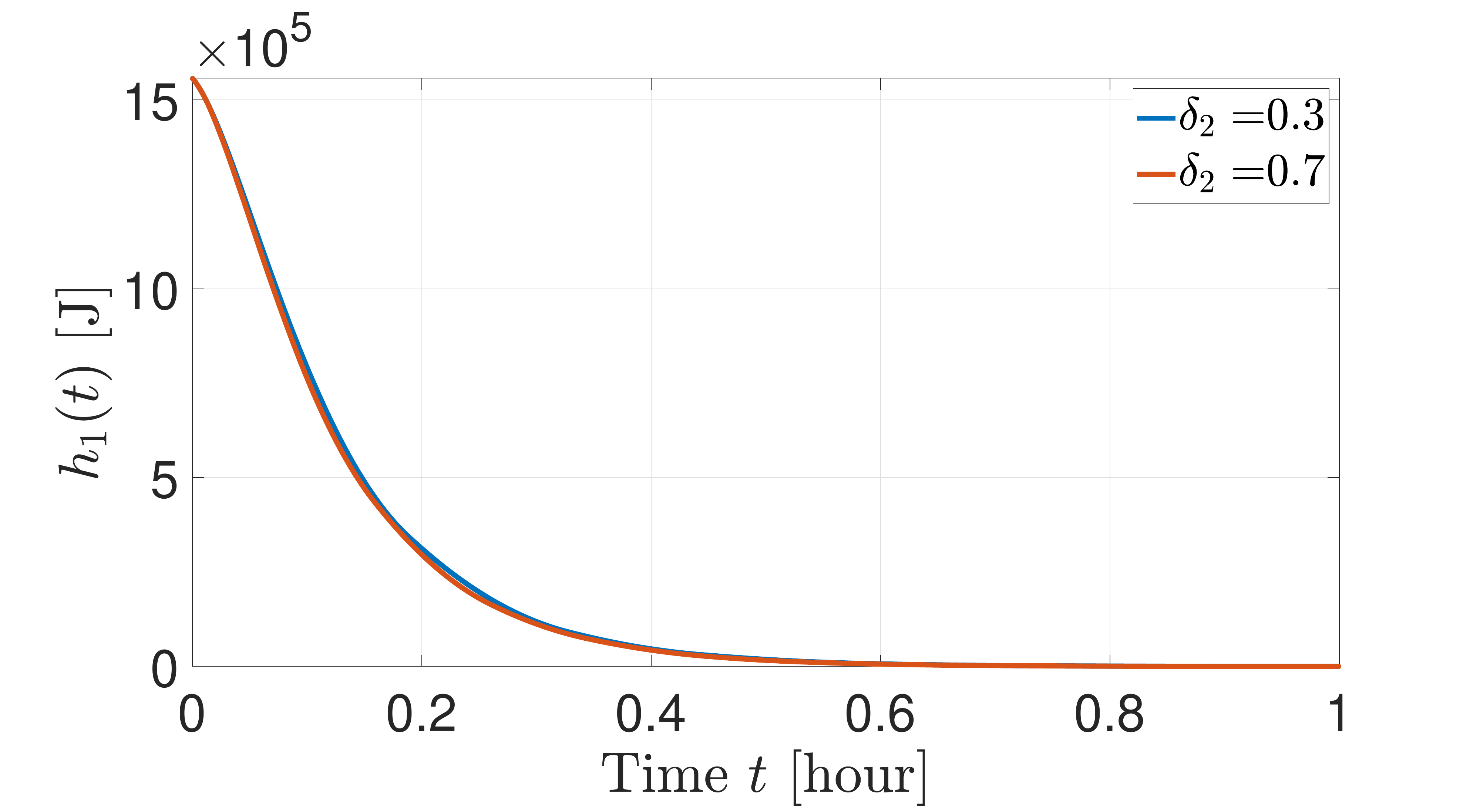}\\
\caption{CBF of energy deficit remains positive. }
\label{fig:h1}
\end{figure}

\begin{figure}[t]
\centering
\includegraphics[width=0.99\linewidth]{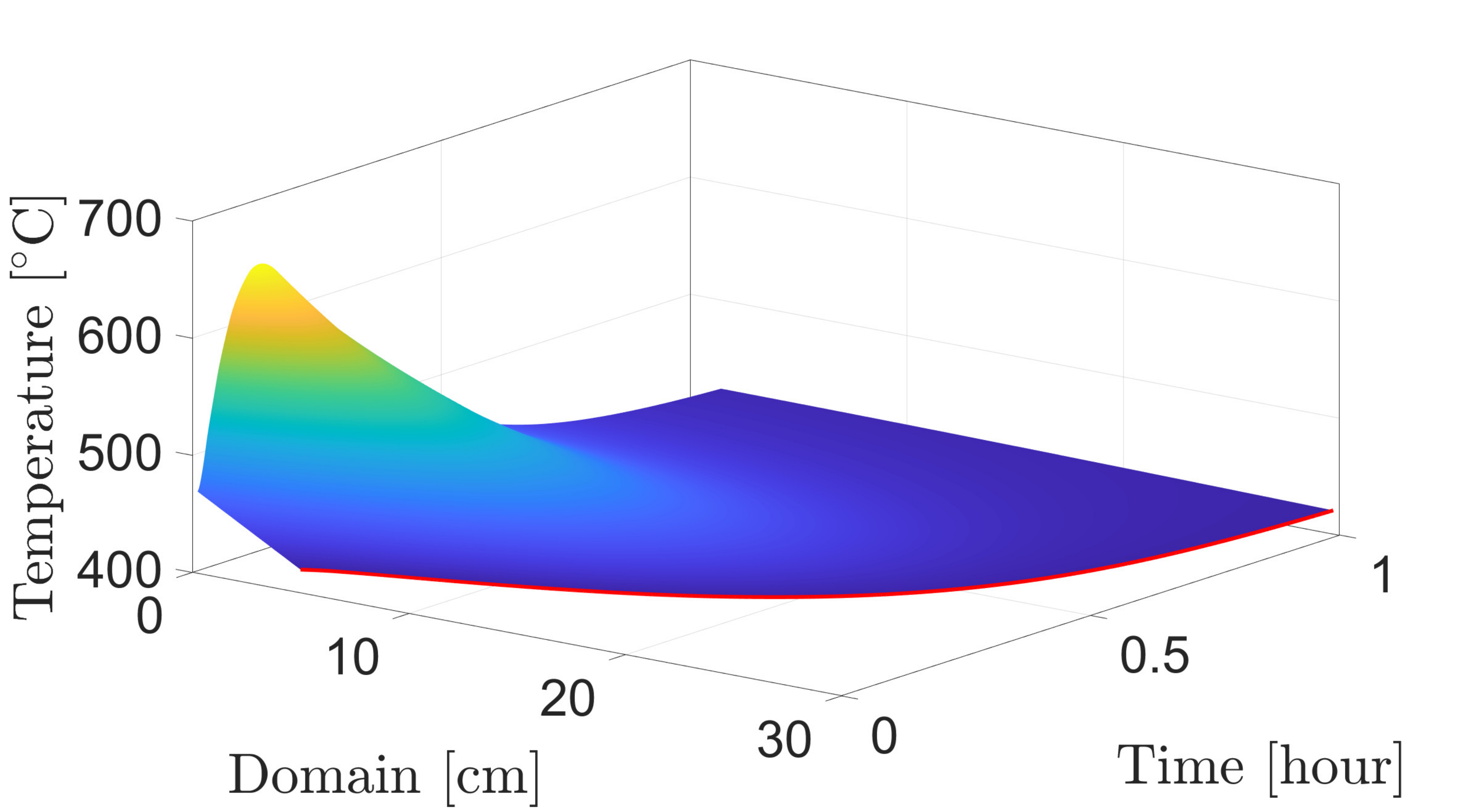}\\
\caption{Liquid temperature, $T(x,t)$ remains above the melting temperature and converges to the melting temperature, $T_m = 420 \ [^\circ C]$. }
\label{fig:hxt}
\end{figure}

\begin{figure}[t]
\centering
\includegraphics[width=0.99\linewidth]{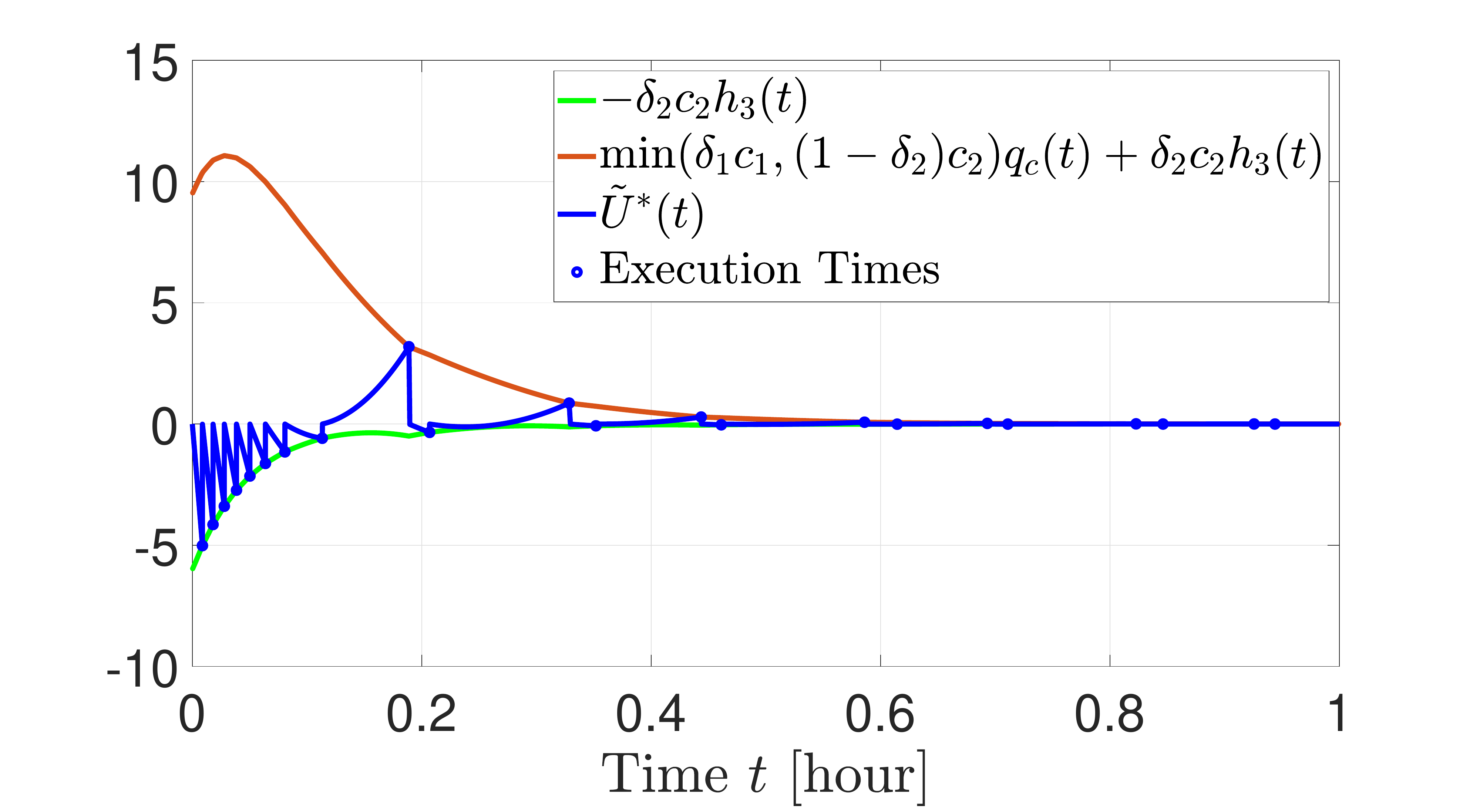}\\
\caption{Trajectories involved in the triggering condition \eqref{eq:cond-safe-stable}. Once the trajectory $\tilde{U}^*(t)$ reaches the lower bound or upper bounds of \eqref{eq:cond-safe-stable} an event is generated, the
control value is updated and $\tilde{U}^*(t)$ is reset to zero.}
\label{fig:Util}
\end{figure}
\section{Conclusion}

This paper has proposed an event-triggered boundary control for the safety and stability of the Stefan PDE system with actuator dynamics. The control law is designed by applying Zero-Order Hold (ZOH) to the nominal continuous-time feedback control law to satisfy both the safety and stability. The event-triggering mechanism is designed so that the safety and stability are still maintained from the positivity of the imposed CBFs and the stability analysis. Future work includes the development of adaptive event-triggered control to identify unknown parameters \cite{wang2022event} and its application to the neuron growth process \cite{demir2021neuron}. 


\bibliographystyle{ieeetr}
\bibliography{ref.bib}


\end{document}